\documentclass[]{amsart}

\usepackage[english]{babel}
\usepackage[utf8]{inputenc}
\usepackage{paralist}
\usepackage{amsmath, amsfonts, amssymb, amsthm}
\usepackage{color}
\usepackage[normalem]{ulem}
\usepackage{tikz}

\newcommand{\CC}{\mathbb{C}}

\newcommand{\NN}{\mathbb{N}}

\newcommand{\ZZ}{\mathbb{Z}}

\newcommand{\Hc}{\mathcal{H}}

\newcommand{\Oc}{\mathcal{O}}

\newcommand{\set}[1]{\left\{ #1 \right\}}
\newcommand{\setb}[1]{\left( #1 \right)}
\newcommand{\abs}[1]{\left| #1 \right|}

\newcommand{\gfr}{\mathfrak{g}}
\newcommand{\bino}[2]{\begin{pmatrix} #1 \\ #2 \end{pmatrix}}

\newtheorem{mymasterthm}{notForUse}
\theoremstyle{definition}

\theoremstyle{plain}
\newtheorem{mylemma}[mymasterthm]{Lemma}
\newtheorem{mythm}[mymasterthm]{Theorem}
\newtheorem{mycorol}[mymasterthm]{Corollary}

\allowdisplaybreaks

\title[A Function Field Variant of Pillai's Problem]{A Function Field Variant of Pillai's Problem}

\author[C. Fuchs]{Clemens Fuchs}
\author[S. Heintze]{Sebastian Heintze}
\thanks{Supported by Austrian Science Fund (FWF): I4406.}

\address{University of Salzburg\newline
	\indent Department of Mathematics\newline
	\indent Hellbrunnerstr. 34 \newline
	\indent A-5020 Salzburg, Austria}
\email{clemens.fuchs@sbg.ac.at, sebastian.heintze@sbg.ac.at}

\begin{document}
	
	\maketitle
	
	
	\begin{abstract}
		In this paper, we consider a variant of Pillai's problem over function fields $ F $ in one variable over $ \CC $. For given simple linear recurrence sequences $ G_n $ and $ H_m $, defined over $ F $ and satisfying some weak conditions, we will prove that the equation $ G_n - H_m = f $ has only finitely many solutions $ (n,m) \in \NN^2 $ for any non-zero $ f \in F $, which can be effectively bounded. Furthermore, we prove that under suitable assumptions there are only finitely many effectively computable $ f $ with more than one representation of the form $ G_n - H_m $.
	\end{abstract}
	
	\section{Introduction}
	
	An about a hundred year old problem going back to Pillai \cite{pillai-1931} considers exponential Diophantine equations of the form
	\begin{equation}
		\label{p8-eq:pillai}
		a^n - b^m = f
	\end{equation}
	for given positive integers $ a,b,f $ to be solved in positive integers $ n,m \geq 2 $.
	If $ a $ and $ b $ are given, then topics of interest are to answer for which $ f $ equation \eqref{p8-eq:pillai} has infinitely many solutions, finitely many solutions or at most one solution in $ (n,m) \in \NN^2 $, respectively, where $ \NN $ denotes the set of natural numbers i.e. positive integers.
	Furthermore, we can ask for a bound on the number or size of solutions $ (n,m) $ if there are only finitely many of them.
	In \cite{pillai-1936} Pillai proved, extending work of Herschfeld \cite{herschfeld-1936} for the case $ a=3, b=2 $, that if $ a $ and $ b $ are coprime, $ a > b \geq 1 $, and $ \abs{f} $ is sufficiently large, then equation \eqref{p8-eq:pillai} has at most one solution.
	He also claimed that \eqref{p8-eq:pillai} can have at most one solution if $ a $ and $ b $ are not coprime, but this is incorrect as was shown by the example $ 6^4 - 3^4 = 1215 = 6^5 - 3^8 $ in \cite{bugeaud-luca-2006}.
	The finiteness of the number of solutions of \eqref{p8-eq:pillai} was already mentioned by P\'olya in \cite{polya-1918}, where instead of Siegel's theorem on integral points on curves the approximation theorem of Thue is used in the proof.
	Bennett proved in \cite{bennett-2001} that for any triple $ (a,b,f) $ of nonzero integers, with $ a,b \geq 2 $, equation \eqref{p8-eq:pillai} has at most two solutions $ (n,m) \in \NN^2 $.
	
	If one allows $ a $ and $ b $ to vary, Pillai conjectured in \cite{pillai-1936} that there are only finitely many solutions with $ n \geq 2, m \geq 2 $. Here for $ f=1 $ we get the famous Catalan conjecture, completely proved by Mih\u ailescu \cite{mihailescu-2004}.
	Other results with varying $ a $ are listed in the paper \cite{bugeaud-luca-2006} already mentioned above. In \cite{luca-2003} Luca used the $ abc $-conjecture to prove that the equation $ p^{n_1} - p^{n_2} = q^{m_1} - q^{m_2} $ has only finitely many solutions $ (p,q,n_1,n_2,m_1,m_2) $ in positive integers, with $ p \neq q $ primes and $ n_1 \neq n_2 $ (see also \cite{waldschmidt-2010} for a quantitative version of Pillai's conjecture that follows from the $ abc $-conjecture).
	For a rather complete historical summary before 2009 on Pillai's problem we refer to \cite{waldschmidt-2010}.
	
	A natural generalisation of this problem is to replace $ a^n $ and $ b^m $ by simple linear recurrence sequences $ A_n = a_1 \alpha_1^n + \cdots + a_d \alpha_d^n $ and $ B_m = b_1 \beta_1^m + \cdots + b_t \beta_t^m $ of integers.
	Pillai's equation is obtained when $ A_n = a^n $ and $ B_m = b^m $.
	Since Waldschmidt's survey \cite{waldschmidt-2010}, a significant number of papers considered this generalization for special recurrences (e.g. $ k $-generalized Fibonacci numbers and powers of $ 2 $ and $ 3 $; just in order to give a concrete reference we mention \cite{ddamulira-gomez-luca-2018}). The authors aimed for and proved complete results in the sense that all exceptions, in which more than one solution exists, were explicitly determined.
	Such a result was proved by Stroeker and Tijdeman (cf. \cite{stroeker-tijdeman-1982}) for the case $ a=3, b=2 $ proving another conjecture by Pillai.
	In \cite{chim-pink-ziegler-2018} Chim, Pink and Ziegler proved that if $ A_n $ and $ B_m $ are strictly increasing in absolute values and have dominant roots $ \alpha $ and $ \beta $, respectively, which are multiplicatively independent, then there exists an effectively computable finite set $ E $ such that $ A_n - B_m = f $ has more than one solution $ (n,m) \in \NN^2 $ if and only if $ f \in E $.
	The independence of the dominant roots is a natural condition since otherwise one can find counterexamples as given in \cite{chim-pink-ziegler-2018}.
	
	In the present paper we consider a function field analogue of the Pillai problem.
	Silverman worked in \cite{silverman-1982} with the Cassels-Catalan equation $ ax^m + by^n = c $ for fixed $ a,b,c $ over function fields.
	We are interested in solutions $ (n,m) \in \NN^2 $ of
	\begin{equation}
		\label{p8-eq:funcfieldeq}
		G_n - H_m = f
	\end{equation}
	where $ G_n $ and $ H_m $ are simple linear recurrence sequences defined over a function field $ F $ in one variable over $ \CC $ and $ f \in F^* $.
	We will prove that under weak assumptions there are only finitely many solutions $ (n,m) \in \NN^2 $ of \eqref{p8-eq:funcfieldeq} for any non-zero $ f \in F $ and that $ n,m $ are bounded by an effectively computable constant, depending only on $ f $, the genus $ \gfr $ of $ F $, and the characteristic roots and coefficients of $ G_n $ and $ H_m $. Moreover, we will show that under some suitable conditions there are only finitely many $ f \in F $, which can be effectively computed, with two distinct representations of the form \eqref{p8-eq:funcfieldeq}.
	
	\section{Notations and Results}
	
	Throughout the paper we denote by $ F $ a function field in one variable over $ \CC $ and by $ \gfr $ the genus of $ F $.
	For the convenience of the reader we give a short wrap-up of the notion of valuations that can e.g. also be found in \cite{fuchs-heintze-p3}:
	For $ c \in \CC $ and $ f(x) \in \CC(x) $, where $ \CC(x) $ is the rational function field over $ \CC $, we denote by $ \nu_c(f) $ the unique integer such that $ f(x) = (x-c)^{\nu_c(f)} p(x) / q(x) $ with $ p(x),q(x) \in \CC[x] $ such that $ p(c)q(c) \neq 0 $. Further we write $ \nu_{\infty}(f) = \deg q - \deg p $ if $ f(x) = p(x) / q(x) $.
	These functions $ \nu : \CC(x) \rightarrow \ZZ $ are up to equivalence all valuations in $ \CC(x) $.
	If $ \nu_c(f) > 0 $, then $ c $ is called a zero of $ f $, and if $ \nu_c(f) < 0 $, then $ c $ is called a pole of $ f $, where $ c \in \CC \cup \set{\infty} $.
	For a finite extension $ F $ of $ \CC(x) $ each valuation in $ \CC(x) $ can be extended to no more than $ [F : \CC(x)] $ valuations in $ F $. This again gives up to equivalence all valuations in $ F $.
	Both, in $ \CC(x) $ as well as in $ F $ the sum-formula
	\begin{equation*}
		\sum_{\nu} \nu(f) = 0
	\end{equation*}
	holds, where the sum is taken over all valuations in the considered function field.
	Moreover, valuations have the properties $ \nu(fg) = \nu(f) + \nu(g) $ and $ \nu(f+g) \geq \min \set{\nu(f), \nu(g)} $ for all $ f,g \in F $.
	
	Our first result is the following theorem which states that any fixed non-zero element $ f \in F $ has only finitely many representations of the form $ G_n - H_m $:
	
	\begin{mythm}
		\label{p8-thm:diffequalc}
		Let $ G_n = a_1 \alpha_1^n + \cdots + a_d \alpha_d^n $ and $ H_m = b_1 \beta_1^m + \cdots + b_t \beta_t^m $ be two simple linear recurrence sequences such that $ a_i, \alpha_i, b_j, \beta_j \in F^* $ for $ i=1,\ldots,d $ and $ j=1,\ldots,t $.
		Assume that no $ \alpha_i $ or $ \beta_j $ and no ratio $ \alpha_i/\alpha_j $ or $ \beta_i/\beta_j $ for $ i \neq j $ lies in $ \CC $.
		Moreover, let $ f \in F^* $ be a given non-zero element.
		Then there exists an effectively computable constant $ C $, which depends only on the $ a_i, \alpha_i, b_j, \beta_j, f $ and $ \gfr $, such that for all $ (n,m) \in \NN^2 $ with $ G_n - H_m = f $ we have
		\begin{equation*}
			\max \setb{n,m} \leq C.
		\end{equation*}
	\end{mythm}
	
	In Corollary 4 in \cite{fuchs-petho-2005} the case of $ G_n - H_m = 0 $ is completely solved. It is proven that there is an effectively computable upper bound for $ \max \setb{n,m} $ unless $ G_n $ and $ H_m $ differ not significantly which is described precisely.
	
	In the special case that $ G_n $ and $ H_m $ are pure powers of non-constant polynomials in $ \CC[x] $ we get:
	
	\begin{mycorol}
		\label{p8-cor:specialcaseconst}
		Let $ p,q,f $ be polynomials in $ \CC[x] $ with $ p,q $ non-constant and $ f $ non-zero. Then for all $ (n,m) \in \NN^2 $ with $ p^n - q^m = f $ we have $ \max \setb{n,m} \leq C $ for
		\begin{equation*}
			C = \frac{1 + \deg p + \deg q + 2\deg f}{\min \setb{\deg p, \deg q}}.
		\end{equation*}
	\end{mycorol}
	
	For our next theorem we will need some further notation.
	In the case that $ G_n = a_1 \alpha_1^n + \cdots + a_d \alpha_d^n $ is a polynomial power sum, i.e. $ \alpha_1, \ldots, \alpha_d \in \CC[x] $, $ \alpha_1 $ is called the dominant root if either $ d=1 $ or $ \deg \alpha_1 > \max_{i=2,\ldots,d} \deg \alpha_i $.
	In the second case this can be rewritten as $ \nu_{\infty}(\alpha_1) < \min_{i=2,\ldots,d} \nu_{\infty}(\alpha_i) $ and since the $ \alpha_i $ are polynomials the inequality $ \nu_{\infty}(\alpha_1) < 0 $ is also satisfied.
	
	Therefore, for more general $ \alpha_1, \ldots, \alpha_d \in F $ and a valuation $ \nu $ in $ F $, we call $ \alpha_1 $ the \emph{$ \nu $-dominant root} if either $ \nu(\alpha_1) < \min \setb{\nu(\alpha_2), \ldots, \nu(\alpha_d), 0} $ or $ d=1 $.
	
	We will now consider elements $ f \in F $ with more than one representation of the form $ G_n - H_m $. Our goal is to prove that, under some not too restrictive conditions, more than one representation is only possible for finitely many $ f $.
	Hence it is obvious that we must exclude situations where $ G_{n_1} = G_{n_2} $ for arbitrary large indices $ n_1, n_2 $.
	Thus we have to assume that there is a bound $ N_0 $ such that for $ n_1, n_2 > N_0 $ we have $ G_{n_1} \neq G_{n_2} $. By throwing away the first $ N_0 $ elements of the recurrence sequence and considering only the remaining ones, we may assume that the map $ n \mapsto G_n $ is injective. We will write \emph{$ G_n $ has no multiple values} for this assumption.
	
	Furthermore, if $ \alpha_1 $ is the $ \nu $-dominant root of $ G_n $, there is an effectively computable bound $ N_1 $ such that for $ n > N_1 $ we have $ \nu(a_1\alpha_1^n) < \min_{i=2,\ldots,d} \nu(a_i\alpha_i^n) $. By the same argument as above we may assume that this inequality holds for all $ n \in \NN $.
	We will refer to this by saying \emph{the $ \nu $-dominant root has immediate effect}.
	
	Last but not least we call two elements $ \alpha, \beta \in F $ \emph{multiplicatively independent} if $ \alpha^r \beta^s \in \CC $ for $ r,s \in \ZZ $ implies that $ r = s = 0 $.
	
	The result is now the following statement which implies that under the given conditions there are only finitely many $ f $ with at least two representations of the form $ G_n - H_m $:
	
	\begin{mythm}
		\label{p8-thm:diffequaldiffdom}
		Let $ G_n = a_1 \alpha_1^n + \cdots + a_d \alpha_d^n $ and $ H_m = b_1 \beta_1^m + \cdots + b_t \beta_t^m $ be two simple linear recurrence sequences such that $ a_i, \alpha_i, b_j, \beta_j \in F^* $ for $ i=1,\ldots,d $ and $ j=1,\ldots,t $.
		Assume that there exists a valuation $ \nu $ in $ F $ such that $ \alpha_1 $ and $ \beta_1 $ are the $ \nu $-dominant roots with immediate effect of $ G_n $ and $ H_m $, respectively, that $ \alpha_1, \beta_1 \notin \CC $, and that $ \alpha_1 $ and $ \beta_1 $ are multiplicatively independent.
		Then there exists an effectively computable constant $ C $, which depends only on the $ a_i, \alpha_i, b_j, \beta_j $ and $ \gfr $, such that for all distinct $ (n_1,m_1), (n_2,m_2) \in \NN^2 $ with $ G_{n_1} - H_{m_1} = G_{n_2} - H_{m_2} $ we have
		\begin{equation*}
			\max \setb{n_1,m_1,n_2,m_2} \leq C.
		\end{equation*}
	\end{mythm}
	
	Note that the assumptions in Theorem \ref{p8-thm:diffequaldiffdom} already imply that $ G_n $ has no multiple values. More precisely it is implied by the assumption $ \alpha_1 \notin \CC $ in the case $ d = 1 $ and by the fact that $ \alpha_1 $ is the $ \nu $-dominant root with immediate effect for $ d > 1 $. The same holds for $ H_m $.
	
	One could ask whether we can relax the dominant root condition.
	Indeed, it is possible to relax the dominant root condition to the cost of requiring more multiplicative independence and prove a similar statement. 
	To fix ideas we restrict ourselves to the polynomial case, i.e. we will assume that $ a_i, \alpha_i \in \CC[x] $.
	Moreover, we will assume (by throwing away the first few elements of the recurrences if necesssary; compare with the dominant root case) that $ \deg \alpha_i > \deg \alpha_j $ implies $ \deg (a_i \alpha_i^n) > \deg (a_j \alpha_j^n) $ for all $ i,j $ and all $ n \in \NN $, and refer to this by saying that \emph{$ G_n $ has weak coefficients}.
	
	Now we describe what we mean by the \emph{relevant set of characteristic roots} for a recurrence with weak coefficients.
	In a preparatory step we order the characteristic roots $ \alpha_i $ such that
	\begin{equation*}
		\deg (a_1 \alpha_1^n) = \cdots = \deg (a_k \alpha_k^n) > \deg (a_{k+1} \alpha_{k+1}^n) \geq \cdots \geq \deg (a_d \alpha_d^n).
	\end{equation*}
	Then we call the set $ R_G = \set{\alpha_1, \ldots, \alpha_k} $ the \emph{relevant set of characteristic roots of $ G_n $}.
	In this language our result is the following statement:
	
	\begin{mythm}
		\label{p8-thm:diffequaldiffrel}
		Let $ G_n = a_1 \alpha_1^n + \cdots + a_d \alpha_d^n $ and $ H_m = b_1 \beta_1^m + \cdots + b_t \beta_t^m $ be two simple linear recurrence sequences such that $ a_i, \alpha_i, b_j, \beta_j \in \CC[x] $ for $ i=1,\ldots,d $ and $ j=1,\ldots,t $.
		Assume that $ G_n $ and $ H_m $ both have weak coefficients.
		Denote by $ R_G $ and $ R_H $ the relevant sets of characteristic roots of $ G_n $ and $ H_m $, respectively, and assume that no element of $ R_G $ or $ R_H $ as well as no quotient of two distinct elements of $ R_G $ or $ R_H $ lies in $ \CC $.
		Moreover, suppose that all pairs in the set $ \set{(\alpha_1, \gamma) : \gamma \in R_H} \cup \set{(\delta, \beta_1) : \delta \in R_G} $ are pairs of multiplicatively independent elements, and that neither $ G_n $ nor $ H_m $ has multiple values.
		Then there exists an effectively computable constant $ C $, which depends only on the $ a_i, \alpha_i, b_j, \beta_j $ and $ \gfr $, such that for all distinct $ (n_1,m_1), (n_2,m_2) \in \NN^2 $ with $ G_{n_1} - H_{m_1} = G_{n_2} - H_{m_2} $ we have
		\begin{equation*}
			\max \setb{n_1,m_1,n_2,m_2} \leq C.
		\end{equation*}
	\end{mythm}
	
	Note that this theorem can be generalized to more general elements in $ F $ if we replace the degree conditions by suitable valuation conditions as we have done in Theorem \ref{p8-thm:diffequaldiffdom}.
	
	\section{Preliminaries}
	
	The proofs in the next section will make use of height functions in function fields. Let us therefore define the height of an element $ f \in F^* $ by
	\begin{equation*}
		\Hc(f) := - \sum_{\nu} \min \setb{0, \nu(f)} = \sum_{\nu} \max \setb{0, \nu(f)}
	\end{equation*}
	where the sum is taken over all valuations on the function field $ F / \CC $. Additionally we define $ \Hc(0) = \infty $.
	This height function satisfies some basic properties that are listed in the lemma below which is proven in \cite{fuchs-karolus-kreso-2019}:
	
	\begin{mylemma}
		\label{p8-lemma:heightproperties}
		Denote as above by $ \Hc $ the height on $ F/\CC $. Then for $ f,g \in F^* $ the following properties hold:
		\begin{enumerate}[a)]
			\item $ \Hc(f) \geq 0 $ and $ \Hc(f) = \Hc(1/f) $,
			\item $ \Hc(f) - \Hc(g) \leq \Hc(f+g) \leq \Hc(f) + \Hc(g) $,
			\item $ \Hc(f) - \Hc(g) \leq \Hc(fg) \leq \Hc(f) + \Hc(g) $,
			\item $ \Hc(f^n) = \abs{n} \cdot \Hc(f) $,
			\item $ \Hc(f) = 0 \iff f \in \CC^* $,
			\item $ \Hc(A(f)) = \deg A \cdot \Hc(f) $ for any $ A \in \CC[T] \setminus \set{0} $.
		\end{enumerate}
	\end{mylemma}
	
	For a finite set $ S $ of valuations on $ F $, we denote by $ \Oc_S^* $ the set of $ S $-units in $ F $, i.e. the set
	\begin{equation*}
		\Oc_S^* = \set{f \in F^* : \nu(f) = 0 \text{ for all } \nu \notin S}.
	\end{equation*}
	
	Furthermore, the following theorem due to Brownawell and Masser plays an essential role within our proofs. It is an immediate consequence of Theorem B in \cite{brownawell-masser-1986}:
	
	\begin{mythm}[Brownawell-Masser]
		\label{p8-thm:brownawellmasser}
		Let $ F/\CC $ be a function field in one variable of genus $ \gfr $. Moreover, for a finite set $ S $ of valuations, let $ u_1,\ldots,u_k $ be $ S $-units and
		\begin{equation*}
			1 + u_1 + \cdots + u_k = 0,
		\end{equation*}
		where no proper subsum of the left hand side vanishes. Then we have
		\begin{equation*}
			\max_{i=1,\ldots,k} \Hc(u_i) \leq \bino{k}{2} \left( \abs{S} + \max \setb{0, 2\gfr-2} \right).
		\end{equation*}
	\end{mythm}
	
	\section{Proofs}
	
	During this section $ C_1, C_2, \ldots $ will denote effectively computable constants. To keep the indices small we will start a new numbering for each proof. Note that therefore there is no dependence between the constants occurring in different proofs.
	We begin with the proof of our first theorem:
	
	\begin{proof}[Proof of Theorem \ref{p8-thm:diffequalc}]
		Let $ G_n, H_m, f $ be as in the theorem. If we insert the sum representations of $ G_n $ and $ H_m $ into the equation $ G_n - H_m = f $, bring all terms to one side, and divide by $ f $, we get
		\begin{equation}
			\label{p8-eq:pr1starteq}
			1 + \frac{b_1}{f} \beta_1^m + \cdots + \frac{b_t}{f} \beta_t^m - \frac{a_1}{f} \alpha_1^n - \cdots - \frac{a_d}{f} \alpha_d^n = 0.
		\end{equation}
		
		Now let $ S $ be a finite set of valuations such that $ f $ and all $ \alpha_i $ and $ a_i $ for $ i = 1,\ldots,d $ as well as all $ \beta_j $ and $ b_j $ for $ j = 1,\ldots,t $ are $ S $-units.
		We define
		\begin{equation*}
			C_1 := \bino{d+t}{2} \left( \abs{S} + \max \setb{0, 2\gfr-2} \right)
		\end{equation*}
		and assume that $ (n,m) \in \NN^2 $ satisfies equation \eqref{p8-eq:pr1starteq}.
		
		Our plan is to apply Theorem \ref{p8-thm:brownawellmasser}.
		Therefore we consider a minimal vanishing subsum of the left hand side of \eqref{p8-eq:pr1starteq}, i.e. no proper sub-subsum of this subsum vanishes, which contains the summand $ 1 $.
		This subsum contains at least one other summand.
		Without loss of generality we may assume that the summand $ -\frac{a_{i_0}}{f} \alpha_{i_0}^n $ is contained therein.
		By Theorem \ref{p8-thm:brownawellmasser} we get the upper bound
		\begin{equation*}
			\Hc \left( -\frac{a_{i_0}}{f} \alpha_{i_0}^n \right) \leq C_1.
		\end{equation*}
		Thus we have
		\begin{align*}
			n \cdot \Hc(\alpha_{i_0}) &= \Hc(\alpha_{i_0}^n) = \Hc \left( -\frac{a_{i_0}}{f} \alpha_{i_0}^n \cdot \frac{-f}{a_{i_0}} \right) \\
			&\leq \Hc \left( -\frac{a_{i_0}}{f} \alpha_{i_0}^n \right) + \Hc\left( \frac{-f}{a_{i_0}} \right) \\
			&\leq C_1 + \max \setb{\max_{i=1,\ldots,d} \Hc\left( \frac{-f}{a_i} \right), \max_{j=1,\ldots,t} \Hc\left( \frac{-f}{b_j} \right)} =: C_2
		\end{align*}
		and
		\begin{equation*}
			n \leq \frac{C_2}{\Hc(\alpha_{i_0})} \leq \frac{C_2}{\min \setb{\min_{i=1,\ldots,d} \Hc(\alpha_i), \min_{j=1,\ldots,t} \Hc(\beta_j)}} =: C_3.
		\end{equation*}
		
		Now there are two possible cases. If also a summand with an $ \beta_j $ is contained in the minimal vanishing subsum with $ 1 $, then the same calculations show that $ m \leq C_3 $ and we are done.
		Otherwise we consider a minimal vanishing subsum of the left hand side of \eqref{p8-eq:pr1starteq} of the form
		\begin{equation*}
			\frac{b_1}{f} \beta_1^m + z_1 + \cdots + z_k = 0.
		\end{equation*}
		After dividing by $ z_1 $ we can apply Theorem \ref{p8-thm:brownawellmasser} to this subsum which yields
		\begin{equation*}
			\Hc \left( \frac{b_1}{fz_1} \beta_1^m \right) \leq C_1.
		\end{equation*}
		
		Let us first assume that $ z_1 = \frac{b_{j_0}}{f} \beta_{j_0}^m $ for $ j_0 \neq 1 $.
		Together with the bound in the last displayed expression we get
		\begin{align*}
			m \cdot \Hc \left( \frac{\beta_1}{\beta_{j_0}} \right) &= \Hc \left( \left( \frac{\beta_1}{\beta_{j_0}} \right)^m \right) = \Hc \left( \frac{b_1}{b_{j_0}} \left( \frac{\beta_1}{\beta_{j_0}} \right)^m \cdot \frac{b_{j_0}}{b_1} \right) \\
			&\leq \Hc \left( \frac{b_1}{b_{j_0}} \left( \frac{\beta_1}{\beta_{j_0}} \right)^m \right) + \Hc \left(\frac{b_{j_0}}{b_1} \right) \\
			&\leq C_1 + \max \setb{\max_{i \neq j} \Hc\left( \frac{a_i}{a_j} \right), \max_{i \neq j} \Hc\left( \frac{b_i}{b_j} \right)} =: C_4
		\end{align*}
		and
		\begin{equation*}
			m \leq \frac{C_4}{\Hc \left( \frac{\beta_1}{\beta_{j_0}} \right)} \leq \frac{C_4}{\min \setb{\min_{i \neq j} \Hc\left( \frac{\alpha_i}{\alpha_j} \right), \min_{i \neq j} \Hc\left( \frac{\beta_i}{\beta_j} \right)}} =: C_5.
		\end{equation*}
		
		Assume now that $ z_1 = -\frac{a_{i_1}}{f} \alpha_{i_1}^n $ for some $ i_1 $. In this situation we end up with the bounds
		\begin{align*}
			m \cdot \Hc(\beta_1) &= \Hc(\beta_1^m) \leq \Hc(b_1 \beta_1^m) + \Hc(b_1) \\
			&\leq \Hc \left( -\frac{b_1 \beta_1^m}{a_{i_1} \alpha_{i_1}^n} \right) + \Hc(a_{i_1} \alpha_{i_1}^n) + \Hc(b_1) \\
			&\leq C_1 + \Hc(a_{i_1}) + n \cdot \Hc(\alpha_{i_1}) + \Hc(b_1) \\
			&\leq C_1 + \max_{i=1,\ldots,d} \Hc(a_i) + \max_{j=1,\ldots,t} \Hc(b_j) \\
			&\hspace*{1cm}+ C_3 \cdot \max \setb{\max_{i=1,\ldots,d} \Hc(\alpha_i), \max_{j=1,\ldots,t} \Hc(\beta_j)} \\
			&=: C_6
		\end{align*}
		and
		\begin{equation*}
			m \leq \frac{C_6}{\Hc(\beta_1)} \leq \frac{C_6}{\min \setb{\min_{i=1,\ldots,d} \Hc(\alpha_i), \min_{j=1,\ldots,t} \Hc(\beta_j)}} =: C_7.
		\end{equation*}
		
		Thus, by putting all things together, we get for the exponential variables $ n,m $ the final bound
		\begin{equation*}
			\max \setb{n,m} \leq \max \setb{C_3, C_5, C_7},
		\end{equation*}
		which proves the theorem.
	\end{proof}
	
	In the special case of pure powers of polynomials the proof as well as the constant becomes much easier:
	
	\begin{proof}[Proof of Corollary \ref{p8-cor:specialcaseconst}]
		Let $ p,q $ be non-constant polynomials in $ \CC[x] $ and $ f $ a non-zero polynomial in $ \CC[x] $.
		For $ (n,m) \in \NN^2 $ with $ p^n - q^m = f $ we get the equation
		\begin{equation*}
			1 + \frac{1}{f} q^m - \frac{1}{f} p^n = 0.
		\end{equation*}
		Since there are only three summands and each of them is non-zero, there cannot be a proper vanishing subsum.
		Let $ S $ be the set containing $ \nu_{\infty} $ as well as the valuations corresponding to the zeros of $ p,q,f $.
		As we have three summands and the genus of $ \CC(x) $ is zero, the bound in Theorem \ref{p8-thm:brownawellmasser} simplifies to $ \abs{S} $, which can be bounded above by
		\begin{equation*}
			\abs{S} \leq 1 + \deg p + \deg q + \deg f.
		\end{equation*}
		Applying Theorem \ref{p8-thm:brownawellmasser} gives then the upper bounds $ \Hc(p^n/f) \leq \abs{S} $ and $ \Hc(q^m/f) \leq \abs{S} $.
		Thus we get
		\begin{align*}
			n \cdot \deg p &= n \cdot \Hc(p) = \Hc(p^n) \leq \Hc(p^n/f) + \Hc(f) \leq \abs{S} + \deg f \\
			&\leq 1 + \deg p + \deg q + 2\deg f
		\end{align*}
		and
		\begin{equation*}
			n \leq \frac{1 + \deg p + \deg q + 2\deg f}{\min \setb{\deg p, \deg q}}.
		\end{equation*}
		The same bound also holds for $ m $, with the same calculations.
	\end{proof}
	
	In preparation of the proof of the other theorems we will formulate and prove a short lemma which will be used several times later on:
	
	\begin{mylemma}
		\label{p8-lemma:quotofindep}
		Let $ \gamma, \delta \in F \setminus \CC $ be multiplicatively independent and $ n,m \in \NN $. Assume that
		\begin{equation}
			\label{p8-eq:heightqtbound}
			\Hc \left( \frac{\gamma^n}{\delta^m} \right) \leq L.
		\end{equation}
		Then there exists an effectively computable constant $ C $, depending only on $ \gamma, \delta, \gfr $ and $ L $, such that
		\begin{equation*}
			\max \setb{n,m} \leq C.
		\end{equation*}
	\end{mylemma}
	
	\begin{proof}
		If $ \gamma $ has a zero that is not a zero of $ \delta $, then we have $ n \leq L $. Analogously, if $ \gamma $ has a pole that is not a pole of $ \delta $, we also have $ n \leq L $.
		If vice versa $ \delta $ has a zero/pole that is not a zero/pole of $ \gamma $, this would imply $ m \leq L $.
		Thus without loss of generality we may assume that either $ n \leq L $ or each zero/pole of $ \gamma $ is also a zero/pole of $ \delta $ and vice versa.
		
		Let us now focus on the second case that $ \gamma $ and $ \delta $ have the same zeros and poles. Since $ \gamma $ and $ \delta $ are multiplicatively independent and non-constant, there exist two valuations $ \nu $ and $ \mu $ such that $ \nu(\gamma) \nu(\delta) \mu(\gamma) \mu(\delta) \neq 0 $ and
		\begin{equation*}
			\frac{\nu(\gamma)}{\nu(\delta)} \neq \frac{\mu(\gamma)}{\mu(\delta)}.
		\end{equation*}
		From inequality \eqref{p8-eq:heightqtbound} we get
		\begin{align*}
			\abs{n \cdot \nu(\gamma) - m \cdot \nu(\delta)} &\leq L \\
			\abs{n \cdot \mu(\gamma) - m \cdot \mu(\delta)} &\leq L.
		\end{align*}
		Therefore we have
		\begin{align*}
			\abs{\frac{\nu(\gamma)}{\nu(\delta)} - \frac{\mu(\gamma)}{\mu(\delta)}} \cdot n &= \abs{n \cdot \frac{\nu(\gamma)}{\nu(\delta)} - n \cdot \frac{\mu(\gamma)}{\mu(\delta)}} \\
			&\leq \abs{n \cdot \frac{\nu(\gamma)}{\nu(\delta)} - m} + \abs{n \cdot \frac{\mu(\gamma)}{\mu(\delta)} - m} \\
			&\leq \frac{L}{\abs{\nu(\delta)}} + \frac{L}{\abs{\mu(\delta)}} =: C_1
		\end{align*}
		and
		\begin{equation*}
			n \leq \frac{C_1}{\abs{\frac{\nu(\gamma)}{\nu(\delta)} - \frac{\mu(\gamma)}{\mu(\delta)}}} =: C_2.
		\end{equation*}
		Hence we have the upper bound $ n \leq \max \setb{L, C_2} =: C_3 $.
		
		Using properties of the height in the same manner as in the proofs above, we get also the upper bounds
		\begin{align*}
			m \cdot \Hc(\delta) &= \Hc(\delta^m) \leq \Hc \left( \frac{\delta^m}{\gamma^n} \right) + \Hc(\gamma^n) \\
			&\leq L + n \cdot \Hc(\gamma) \leq  L + C_3 \cdot \Hc(\gamma) =: C_4
		\end{align*}
		and
		\begin{equation*}
			m \leq \frac{C_4}{\Hc(\delta)}.
		\end{equation*}
		This proves the lemma.
	\end{proof}
	
	Now we will use this lemma to prove our second theorem:
	
	\begin{proof}[Proof of Theorem \ref{p8-thm:diffequaldiffdom}]
		Let $ (n_1,m_1), (n_2,m_2) \in \NN^2 $ be two distinct pairs with $ G_{n_1} - H_{m_1} = G_{n_2} - H_{m_2} $.
		Since neither $ G_n $ nor $ H_m $ has multiple values we have $ n_1 \neq n_2 $ and $ m_1 \neq m_2 $.
		We write $ N = \max \setb{n_1,n_2} $ and $ M = \max \setb{m_1,m_2} $.
		If we insert the sum representations into $ G_{n_1} - H_{m_1} = G_{n_2} - H_{m_2} $ and bring all terms to one side, we get
		\begin{multline}
			\label{p8-eq:pr2starteq}
			a_1 \alpha_1^{n_1} + \cdots + a_d \alpha_d^{n_1} - b_1 \beta_1^{m_1} - \cdots - b_t \beta_t^{m_1} \\- a_1 \alpha_1^{n_2} - \cdots - a_d \alpha_d^{n_2} + b_1 \beta_1^{m_2} + \cdots + b_t \beta_t^{m_2} = 0.
		\end{multline}
		
		Let $ S $ be a finite set of valuations such that all $ \alpha_i $ and $ a_i $ for $ i = 1,\ldots,d $ as well as all $ \beta_j $ and $ b_j $ for $ j = 1,\ldots,t $ are $ S $-units.
		
		Now we differ between four cases. Firstly, we assume $ d = t = 1 $.
		Then equation \eqref{p8-eq:pr2starteq} reduces to
		\begin{equation}
			\label{p8-eq:orderone}
			a_1 \alpha_1^{n_1} - b_1 \beta_1^{m_1} - a_1 \alpha_1^{n_2} + b_1 \beta_1^{m_2} = 0.
		\end{equation}
		If there is no proper vanishing subsum, then we divide by $ a_1 \alpha_1^N $ and apply Theorem \ref{p8-thm:brownawellmasser}. Thus there is an effectively computable constant $ C_1 $ such that
		\begin{equation*}
			\Hc \left( \frac{b_1 \beta_1^M}{a_1 \alpha_1^N} \right) \leq C_1.
		\end{equation*}
		Therefore we have
		\begin{equation*}
			\Hc \left( \frac{\beta_1^M}{\alpha_1^N} \right) \leq \Hc \left( \frac{b_1 \beta_1^M}{a_1 \alpha_1^N} \right) + \Hc \left( \frac{b_1}{a_1} \right) \leq C_1 + \Hc \left( \frac{b_1}{a_1} \right) =: C_2
		\end{equation*}
		and by Lemma \ref{p8-lemma:quotofindep}
		\begin{equation*}
			\max \setb{n_1,m_1,n_2,m_2} = \max \setb{N,M} \leq C_3.
		\end{equation*}
		
		Otherwise we can split equation \eqref{p8-eq:orderone} into two vanishing subsums of the shape
		\begin{align*}
			a_1 \alpha_1^{k_1} \pm b_1 \beta_1^{l_1} &= 0 \\
			a_1 \alpha_1^{k_2} \pm b_1 \beta_1^{l_2} &= 0
		\end{align*}
		for $ \set{k_1,k_2} = \set{n_1,n_2} $ and $ \set{l_1,l_2} = \set{m_1,m_2} $.
		Note that $ a_1 \alpha_1^{n_1} - a_1 \alpha_1^{n_2} = 0 $ is not possible since $ \alpha_1 \notin \CC $.
		Thus we have
		\begin{equation*}
			\Hc \left( \frac{\beta_1^{l_1}}{\alpha_1^{k_1}} \right) = \Hc \left( \frac{a_1}{b_1} \right) = \Hc \left( \frac{\beta_1^{l_2}}{\alpha_1^{k_2}} \right)
		\end{equation*}
		and again by Lemma \ref{p8-lemma:quotofindep}
		\begin{equation*}
			\max \setb{n_1,m_1,n_2,m_2} = \max \setb{k_1,l_1,k_2,l_2} \leq C_4.
		\end{equation*}
		
		Secondly, we assume $ d = 1 $ and $ t > 1 $.
		Let $ \set{M,m_0} = \set{m_1,m_2} $ and $ \set{k_1,k_2} = \set{n_1,n_2} $.
		Since $ \beta_1 $ is the $ \nu $-dominant root with immediate effect of $ H_m $ we have
		\begin{align*}
			\nu(b_1 \beta_1^M) &< \nu(b_j \beta_j^M) \\
			\nu(b_1 \beta_1^M) &< \nu(b_1 \beta_1^{m_0}) < \nu(b_j \beta_j^{m_0})
		\end{align*}
		for $ j=2,\ldots,t $.
		We claim that there is a minimal vanishing subsum of \eqref{p8-eq:pr2starteq} containing $ b_1 \beta_1^M $ and $ a_1 \alpha_1^{k_1} $. If this would not be so, then $ b_1 \beta_1^M $ could be written as a sum of elements with $ \nu $-valuation strictly greater than $ \nu(b_1 \beta_1^M) $, but this is impossible.
		Hence we divide this minimal vanishing subsum by $ b_1 \beta_1^M $ and the application of Theorem \ref{p8-thm:brownawellmasser} gives us
		\begin{equation*}
			\Hc \left( \frac{a_1 \alpha_1^{k_1}}{b_1 \beta_1^M} \right) \leq C_5.
		\end{equation*}
		As we have seen above this yields under use of Lemma \ref{p8-lemma:quotofindep}
		\begin{equation*}
			\max \setb{k_1,m_1,m_2} \leq C_6.
		\end{equation*}
		The summand $ a_1 \alpha_1^{k_2} $ must be part of a minimal vanishing subsum with at least one other summand $ \omega $.
		Since the exponential variable occurring in $ \omega $ is among $ k_1,m_1,m_2 $, the height $ \Hc(\omega) $ can be bounded by an effectively computable constant.
		Therefore we have
		\begin{align*}
			k_2 \cdot \Hc(\alpha_1) &= \Hc(\alpha_1^{k_2}) \leq \Hc(a_1 \alpha_1^{k_2}) + \Hc(a_1) \\
			&\leq \Hc \left( \frac{a_1 \alpha_1^{k_2}}{\omega} \right) + \Hc(\omega) + \Hc(a_1) \leq C_7
		\end{align*}
		because the height of the quotient in the last line is bounded by Theorem \ref{p8-thm:brownawellmasser}.
		Altogether we have
		\begin{equation*}
			\max \setb{n_1,m_1,n_2,m_2} = \max \setb{k_2,k_1,m_1,m_2} \leq C_8.
		\end{equation*}
		The third case that $ d > 1 $ and $ t = 1 $ is handled analogously.
		
		Finally, assume that $ d > 1 $ and $ t > 1 $.
		Let $ \set{M,m_0} = \set{m_1,m_2} $ and $ \set{N,n_0} = \set{n_1,n_2} $.
		Since $ \alpha_1 $ and $ \beta_1 $ are the $ \nu $-dominant roots with immediate effect of $ G_n $ and $ H_m $, respectively, we have
		\begin{align*}
			\nu(a_1 \alpha_1^N) &< \nu(a_i \alpha_i^N) \\
			\nu(a_1 \alpha_1^N) &< \nu(a_1 \alpha_1^{n_0}) < \nu(a_i \alpha_i^{n_0})\\
			\nu(b_1 \beta_1^M) &< \nu(b_j \beta_j^M) \\
			\nu(b_1 \beta_1^M) &< \nu(b_1 \beta_1^{m_0}) < \nu(b_j \beta_j^{m_0})
		\end{align*}
		for $ i=2,\ldots,d $ and $ j=2,\ldots,t $.
		Note that no summand of a vanishing subsum of the left hand side of \eqref{p8-eq:pr2starteq} can have $ \nu $-valuation strictly smaller than each of the other summands.
		Otherwise this element could be written as a sum of elements with $ \nu $-valuation strictly greater than its own $ \nu $-valuation, which is impossible.
		Thus it must hold that $ \nu(a_1 \alpha_1^N) = \nu(b_1 \beta_1^M) $.
		Moreover, $ a_1 \alpha_1^N $ and $ b_1 \beta_1^M $ are in the same minimal vanishing subsum.
		Dividing this minimal vanishing subsum by $ a_1 \alpha_1^N $ and applying Theorem \ref{p8-thm:brownawellmasser} yields
		\begin{equation*}
			\Hc \left( \frac{b_1 \beta_1^M}{a_1 \alpha_1^N} \right) \leq C_9.
		\end{equation*}
		By Lemma \ref{p8-lemma:quotofindep} this implies
		\begin{equation*}
			\max \setb{n_1,m_1,n_2,m_2} = \max \setb{N,M} \leq C_{10}.
		\end{equation*}
		Thus the theorem is proven.
	\end{proof}
	
	Finally, by using similar ideas, we prove our last theorem:
	
	\begin{proof}[Proof of Theorem \ref{p8-thm:diffequaldiffrel}]
		As in the proof of Theorem \ref{p8-thm:diffequaldiffdom} we let $ S $ be a finite set of valuations such that all $ \alpha_i $ and $ a_i $ for $ i = 1,\ldots,d $ as well as all $ \beta_j $ and $ b_j $ for $ j = 1,\ldots,t $ are $ S $-units.
		Let again $ (n_1,m_1), (n_2,m_2) \in \NN^2 $ be two distinct pairs with $ G_{n_1} - H_{m_1} = G_{n_2} - H_{m_2} $.
		Since neither $ G_n $ nor $ H_m $ has multiple values we have $ n_1 \neq n_2 $ and $ m_1 \neq m_2 $.
		We write $ N = \max \setb{n_1,n_2} $ and $ M = \max \setb{m_1,m_2} $, and consider once again equation \eqref{p8-eq:pr2starteq}.
		
		We differ between four cases. The case $ d = t = 1 $ is covered by Theorem \ref{p8-thm:diffequaldiffdom}.
		Now we consider the case $ d = 1 $ and $ t > 1 $.
		In a minimal vanishing subsum containing $ b_1 \beta_1^M $ for degree reasons there must be also contained either another summand $ b_{j_1} \beta_{j_1}^M $ for $ \beta_{j_1} \in R_H $ or $ a_1 \alpha_1^{k_1} $ for $ k_1 \in \set{n_1,n_2} $.
		
		In the first subcase we divide this minimal vanishing subsum by $ b_1 \beta_1^M $ and Theorem \ref{p8-thm:brownawellmasser} gives an effectively computable constant $ C_1 $ such that
		\begin{equation*}
			\Hc \left( \frac{b_{j_1} \beta_{j_1}^M}{b_1 \beta_1^M} \right) \leq C_1.
		\end{equation*}
		Hence we have
		\begin{align*}
			M \cdot \Hc \left( \frac{\beta_{j_1}}{\beta_1} \right) &= \Hc \left( \left( \frac{\beta_{j_1}}{\beta_1} \right)^M \right) = \Hc \left( \frac{b_{j_1}}{b_1} \left( \frac{\beta_{j_1}}{\beta_1} \right)^M \cdot \frac{b_1}{b_{j_1}} \right) \\
			&\leq \Hc \left( \frac{b_{j_1}}{b_1} \left( \frac{\beta_{j_1}}{\beta_1} \right)^M \right) + \Hc \left(\frac{b_1}{b_{j_1}} \right) \leq C_2
		\end{align*}
		and thus $ M \leq C_3 $.
		
		In the second subcase we divide this minimal vanishing subsum also by $ b_1 \beta_1^M $ and Theorem \ref{p8-thm:brownawellmasser} gives an effectively computable constant $ C_4 $ such that
		\begin{equation*}
			\Hc \left( \frac{a_1 \alpha_1^{k_1}}{b_1 \beta_1^M} \right) \leq C_4.
		\end{equation*}
		Since $ \alpha_1 $ and $ \beta_1 $ are multiplicatively independent by assumption, we get under use of Lemma \ref{p8-lemma:quotofindep} the bound $ M \leq C_5 $.
		
		Thus we have a bound $ \max \setb{m_1, m_2} = M \leq C_6 $.
		The summand $ a_1 \alpha_1^{n_1} $ must be contained in a minimal vanishing subsum with at least one element $ \omega $ of the form $ b_j \beta_j^{m_1} $ or $ b_j \beta_j^{m_2} $ because $ a_1 \alpha_1^{n_1} - a_1 \alpha_1^{n_2} \neq 0 $.
		Since the exponential variable occurring in $ \omega $ is among $ m_1, m_2 $, the height $ \Hc(\omega) $ can be bounded by an effectively computable constant.
		Therefore we have
		\begin{align*}
			n_1 \cdot \Hc(\alpha_1) &= \Hc(\alpha_1^{n_1}) \leq \Hc(a_1 \alpha_1^{n_1}) + \Hc(a_1) \\
			&\leq \Hc \left( \frac{a_1 \alpha_1^{n_1}}{\omega} \right) + \Hc(\omega) + \Hc(a_1) \leq C_7
		\end{align*}
		because the height of the quotient in the last line is bounded by Theorem \ref{p8-thm:brownawellmasser}.
		For $ n_2 $ we get an analogous bound.
		Altogether we have
		\begin{equation*}
			\max \setb{n_1,n_2,m_1,m_2} \leq C_8.
		\end{equation*}
		The third case that $ d > 1 $ and $ t = 1 $ is handled analogously.
		
		Finally, assume that $ d > 1 $ and $ t > 1 $.
		Without loss of generality we may assume that $ \deg (a_1 \alpha_1^N) \geq \deg (b_1 \beta_1^M) $.
		Consider a minimal vanishing subsum of equation \eqref{p8-eq:pr2starteq} containing $ a_1 \alpha_1^N $.
		For degree reasons in this subsum must be also contained either a summand of the form $ b_{j_2} \beta_{j_2}^M $ for $ \beta_{j_2} \in R_H $ or another summand of the form $ a_{i_2} \alpha_{i_2}^N $ for $ \alpha_{i_2} \in R_G $.
		
		In the first case we divide this minimal vanishing subsum by $ a_1 \alpha_1^N $ and Theorem \ref{p8-thm:brownawellmasser} yields
		\begin{equation*}
			\Hc \left( \frac{b_{j_2} \beta_{j_2}^M}{a_1 \alpha_1^N} \right) \leq C_9.
		\end{equation*}
		As $ \alpha_1 $ and $ \beta_{j_2} $ are multiplicatively independent by assumption, we get under use of Lemma \ref{p8-lemma:quotofindep} the bound
		\begin{equation*}
			\max \setb{n_1,n_2,m_1,m_2} = \max \setb{N,M} \leq C_{10}.
		\end{equation*}
		
		In the second case we divide this minimal vanishing subsum as well by $ a_1 \alpha_1^N $ and Theorem \ref{p8-thm:brownawellmasser} yields
		\begin{equation*}
			\Hc \left( \frac{a_{i_2} \alpha_{i_2}^N}{a_1 \alpha_1^N} \right) \leq C_{11}.
		\end{equation*}
		As we have seen above this gives a bound $ N \leq C_{12} $.
		Then we consider a minimal vanishing subsum of \eqref{p8-eq:pr2starteq} containing $ b_1 \beta_1^M $.
		If it contains another summand of the form $ b_{j_3} \beta_{j_3}^M $ for $ \beta_{j_3} \in R_H $, then we get an upper bound $ M \leq C_{13} $ in the same manner as for $ N $.
		Otherwise this minimal vanishing subsum must contain a summand $ a_{i_3} \alpha_{i_3}^{k_2} $ for $ k_2 \in \set{n_1, n_2} $ and $ i_3 \in \set{1,\ldots,d} $.
		Here we get a bound
		\begin{equation*}
			\Hc \left( \frac{b_1 \beta_1^M}{a_{i_3} \alpha_{i_3}^{k_2}} \right) \leq C_{14}
		\end{equation*}
		by Theorem \ref{p8-thm:brownawellmasser}.
		We have already seen that such an inequality ends up in a bound $ M \leq C_{15} $ since $ k_2 $ is bounded.
		Thus we have
		\begin{equation*}
			\max \setb{n_1,n_2,m_1,m_2} = \max \setb{N,M} \leq C_{16}
		\end{equation*}
		and the theorem is proven.
	\end{proof}


\begin{thebibliography}{99}
		\bibitem{bennett-2001}
			\textsc{M. A. Bennett},
			On some exponential equations of S.S. Pillai,
			\textit{Canad. J. Math.} \textbf{53} (2001), no. 5, 897-922.
		\bibitem{brownawell-masser-1986}
			\textsc{W. D. Brownawell and D. W. Masser},
			Vanishing sums in function fields,
			\textit{Math. Proc. Camb. Phil. Soc.} \textbf{100} (1986), no. 3, 427-434.
		\bibitem{bugeaud-luca-2006}
			\textsc{Y. Bugeaud and F. Luca},
			On Pillai's Diophantine equation,
			\textit{New York J. Math.} \textbf{12} (2006), 193-217.
		\bibitem{chim-pink-ziegler-2018}
			\textsc{K. C. Chim, I. Pink and V. Ziegler},
			On a Variant of Pillai's Problem II,
			\textit{J. Number Theory} \textbf{183} (2018), 269-290.
		\bibitem{ddamulira-gomez-luca-2018}
			\textsc{M. Ddamulira, C. A. Gomez and F. Luca},
			On a problem of Pillai with $ k $-generalized Fibonacci numbers and powers of $ 2 $,
			\textit{Monatsh. Math.} \textbf{187} (2018), no. 4, 635-664.
		\bibitem{fuchs-heintze-p3}
			\textsc{C. Fuchs and S. Heintze},
			On the Growth of Linear Recurrences in Function Fields,
			\textit{preprint}, arXiv:2006.11074.
		\bibitem{fuchs-karolus-kreso-2019}
			\textsc{C. Fuchs, C. Karolus and D. Kreso},
			Decomposable polynomials in second order linear recurrence sequences,
			\textit{Manuscripta Math.} \textbf{159(3)} (2019), 321-346.
		\bibitem{fuchs-petho-2005}
			\textsc{C. Fuchs and A. Peth\H{o}},
			Effective bounds for the zeros of linear recurrences in function fields,
			\textit{J. Theor. Nombres Bordeaux} \textbf{17} (2005), 749-766.
		\bibitem{herschfeld-1936}
			\textsc{A. Herschfeld},
			The equation $ 2^x-3^y=d $,
			\textit{Bull. Amer. Math. Soc.} \textbf{42} (1936), 231-234.
		\bibitem{luca-2003}
			\textsc{F. Luca},
			On the Diophantine equation $ p^{x_1} - p^{x_2} = q^{y_1} - q^{y_2} $,
			\textit{Indag. Math. (N.S.)} \textbf{14} (2003), no. 2, 207-222.
		\bibitem{mihailescu-2004}
			\textsc{P. Mih\u ailescu},
			Primary cyclotomic units and a proof of Catalan's conjecture,
			\textit{J. Reine Angew. Math.} \textbf{572} (2004), 167-195.
		\bibitem{pillai-1931}
			\textsc{S. S. Pillai},
			On the inequality $ 0 < a^x - b^y \leq n $,
			\textit{J. Indian Math. Soc.} \textbf{19} (1931), 1-11.
		\bibitem{pillai-1936}
			\textsc{S. S. Pillai},
			On $ a^x - b^y = c $,
			\textit{J. Indian Math. Soc. (N.S.)} \textbf{2} (1936), 119-122.
		\bibitem{polya-1918}
			\textsc{G. P\'olya},
			Zur arithmetischen Untersuchung der Polynome,
			\textit{Math. Z.} \textbf{1} (1918), 143-148.
		\bibitem{silverman-1982}
			\textsc{J. H. Silverman},
			The Catalan equation over function fields,
			\textit{Trans. Amer. Math. Soc.} \textbf{273} (1982), 201-205.
		\bibitem{stroeker-tijdeman-1982}
			\textsc{R. J. Stroeker and R. Tijdeman},
			Diophantine equations,
			in \textit{Computational methods in number theory}, Part II, vol. 155 of Math. Centre Tracts, Math. Centrum, Amsterdam, 1982, 321-369.
		\bibitem{waldschmidt-2010}
			\textsc{M. Waldschmidt},
			Perfect Powers: Pillai's works and their developments,
			in \textit{Collected works of S. Sivasankaranarayana Pillai, Eds. R. Balasubramanian and R. Thangadurai}, Collected Works Series, no. 1, Ramanujan Mathematical Society, Mysore, 2010,
			arXiv:0908.4031.
	\end{thebibliography}
\end{document}